\documentclass{amsart}
\usepackage{amsmath, amscd, amssymb, latexsym, hhline, euscript, mathrsfs, bbm}
\usepackage[all]{xy}
\usepackage{enumitem}
\usepackage[usenames]{color} 

\def\1{{\mathbbm{1}}}
\def\a{{\alpha}}
\def\b{{\beta}}
\def\g{{\gamma}}
\def\e{{\epsilon}}
\def\w{{\omega}}
\def\k{{\Bbbk}}

\def\t{{\tau}}

\def\G{{\Gamma}}

\def\BC{{\mathbb{C}}}

\def\BZ{{\mathbb{Z}}}

\def\AA{\mathcal{A}}
\def\CC{\mathcal{C}}

\def\BB{\mathcal{B}}

\def\SS{\mathcal{S}}
\def\Cleft{\text{Cleft}}

\newcommand\SC{\operatorname{SC}}
\newcommand\Rep{\operatorname{Rep}}

\newcommand\id{\operatorname{id}}

\def\to{\rightarrow}
\def\dim{{\mbox{\rm dim}}}

\def\Hom{{\mbox{\rm Hom}}}

\newcommand\enumeri[1]{\begin{enumerate}[label=\rm(\roman*), leftmargin=*] #1 \end{enumerate}}

\newcommand\inv{^{-1}}
\newcommand\x{^{\times}}

\newcommand\ol[1]{\overline{#1}}
\newcommand\ul[1]{\underline{#1}}
\renewcommand\o{\otimes}

\newtheorem{thm}{Theorem}[section]

\newtheorem{prop}[thm]{Proposition}
\newtheorem{lem}[thm]{Lemma}

\theoremstyle{definition}
\newtheorem{defn}[thm]{Definition}
\newtheorem{example}[thm]{Example}

\theoremstyle{remark}
\newtheorem{remark}[thm]{Remark}

\makeatletter
\def\namelabel#1#2{\@bsphack
  \protected@write\@auxout{}%
         {\string\newlabel{#1.nme}{{#2}{#2}}}%
  \@esphack}

\makeatother

\title{Cleft Extensions and Quotients of Twisted Quantum Doubles }
\author{Geoffrey Mason}
\address{Department of Mathematics, University of Calfornia at Santa Cruz, CA 95064}
 \email{gem@ucsc.edu}
 \thanks{Research of the first author was partially supported by NSA and NSF}
\author{Siu-Hung Ng}
\address{Department of Mathematics, Louisiana State University, Baton Rouge, LA 70803}
\address{Department of Mathematics, Iowa State University, Ames, IA 50011}
 \email{rng@iastate.edu}
\thanks{The second author was supported by NSF DMS1303253.}
\begin{document}
\maketitle
\begin{abstract}
Given a pair of finite groups $F, G$ and a normalized 3-cocycle $\w$ of $G$, where $F$ acts on $G$ as automorphisms, we consider
quasi-Hopf algebras defined as a cleft extension $\k^G_\w\#_c\,\k F$ where $c$ denotes some suitable cohomological
data.\ When $F\rightarrow \ol{F}:=F/A$ is a quotient of $F$ by a central subgroup $A$ acting trivially on $G$,
we give necessary and sufficient conditions for the existence of a surjection of quasi-Hopf  algebras and cleft extensions of the type $\k^G_\w\#_c\, \k F\rightarrow \k^G_\w\#_{\ol{c}} \, \k \ol{F}$.\ Our construction is particularly natural
 when $F=G$ acts on $G$ by conjugation, and $\k^G_\w\#_c \k G$ is a twisted quantum double $D^{\omega}(G)$.\ In this case, we give necessary and sufficient conditions that Rep($\k^G_\w\#_{\ol{c}} \, \k \ol{G}$) is a modular tensor category.

\end{abstract}
\section{Introduction}
Given finite groups $F, G$ with a right action of $F$ on $G$ as \emph{automorphisms}, one can
form the \emph{cross product} $\k^G\# \k F$, which is naturally a Hopf algebra and a \emph{trivial cleft extension}.\ Moreover, given a normalized 3-cocycle $\w$ of $G$ and
suitable cohomological data $c$, this construction can be `twisted' to yield a quasi-Hopf  algebra
$\k^G_\w\#_c\, \k G$.\ (Details are deferred to  Section \ref{QHs}.)\ For a surjection of groups
$\pi: F\rightarrow \ol{F}$ such that $\ker \pi$ acts trivially on $G$, we consider
 the possibility of constructing another quasi-Hopf algebra $\k^G_\w\#_{\ol{c}}\, \k \ol{F}$ (for suitable data
 $\ol{c}$)  for which there is a `natural' surjection of quasi-Hopf  algebras
 $f: \k^G_\w\#_c\, \k F\rightarrow \k^G_\w\#_{\ol{c}}\, \k \ol{F}$.\ In general such a construction  is not possible.\
 The main result of the present paper (Theorem \ref{thmcext}) gives \emph{necessary and sufficient} conditions for
 the
 existence of  $\k^G_\w\#_{\ol{c}}\, \k \ol{F}$ and $f$ in the important case when $\ker \pi$ is contained in the
 \emph{center} $Z(F)$ of $F$.\ The conditions involve rather subtle cohomological conditions on
 $\ker \pi$; when they are satisfied we obtain interesting new quasi-Hopf  algebras.

 \medskip
 A special case of this construction applies to the twisted quantum double $D^{\omega}(G)$ \cite{DPR},
 where $F=G$ acts on $G$ by conjugation and the condition that $\ker \pi$ acts trivially on $G$
 is \emph{equivalent} to the centrality of $\ker \pi$.\ In this case, we obtain quotients
 $\k^G_\w\#_{\ol{c}}\,\k \ol{G}$ of the twisted quantum double whenever the relevant cohomological conditions hold.\
 Related objects were considered in \cite{GM}, and in the case that $\pm I\in G\subseteq SU_2(\mathbb{C})$ the fusion
 rules were investigated.\ In fact, we can prove
 that the modular data of each of the orbifold conformal field theories $V_{\widehat{\frak{sl}}_2}^{\ol{G}}$, where
 $\widehat{\frak{sl}}_2$ is the level $1$ affine Kac-Moody Lie algebra of type
 $\frak{sl}_2$ and $\ol{G}=G/\pm I$, are reproduced
 by the modular data of  $\k^G_\w\#_{\ol{c}}\,\k\ol{G}$ for suitable choices of cohomological data $\w$ and $\ol{c}$.\
 This result will be appear elsewhere.

 \medskip
 The paper is organized as follows.\ In Section \ref{QHs} we introduce a category
 associated to a fixed quasi-Hopf algebra $\k_{\omega}^G$ whose objects are the cleft extensions we are interested
 in.\ In Section \ref{CQs} we focus on central extensions and establish the main existence result (Theorem \ref{thmcext}).\ In Sections \ref{DG} and \ref{Smodtc} we consider the special case of twisted quantum doubles.\ The main result here (Theorem \ref{t:modular}) gives necessary and sufficient conditions for
  Rep($\k^G_\w\#_{\ol{c}}\,\k\ol{G}$) to be a modular tensor category.

\section{Quasi-Hopf algebras and cleft extensions}\label{QHs}
A quasi-Hopf algebra is a tuple $(H, \Delta, \epsilon, \phi, \alpha, \beta, S)$ consisting of
a quasi-bialgebra $(H, \Delta, \epsilon, \phi)$ together with an antipode $S$ and distinguished elements
$\alpha, \beta \in H$ which together satisfy various consistency conditions.\
See, for example, \cite{D}, \cite{K}, \cite{MN1}. A Hopf algebra is a quasi-Hopf algebra with
$\alpha=\beta=1$ and trivial Drinfel'd associator $\phi=1\otimes1\otimes 1$. As long as
$\alpha$ is invertible,  $(H, \Delta, \epsilon, \phi, 1, \beta\alpha^{-1}, S_{\alpha})$ is also a quasi-Hopf algebra
for some antipode $S_\a$ (\cite{D}).\ All of the examples of quasi-Hopf algebras in the present paper,
constructed from data associated to a group, will satisfy the condition $\a =1$.

\medskip
Suppose that $G$ is a finite group, $\k$ a field, and $\omega \in Z^3(G, \k\x)$ a
normalized (multiplicative) $3$-cocycle.\ There are several well-known quasi-Hopf algebras associated to this data.\
The group algebra $\k G$ is a Hopf algebra, whence
it is a quasi-Hopf algebra too.\ The dual group algebra is also a quasi-Hopf algebra
$\k^G_\omega$ when equipped with the Drinfel'd associator
\begin{equation}\label{eq:associator}
\phi=\sum_{a,b, c \in G} \w(a,b,c)\inv e_a \o e_b \o e_c,
\end{equation}
where $\{e_a\mid a \in G\}$ is the basis of $\k^G$ dual to the basis of group elements $\{a\mid a \in G\}$ in $\k
G$.\ Here, $\b = \sum_{a \in G} \w(a, a\inv, a) e_a$ and $S(a)= a\inv$ for $a \in G$.\ In particular, $\k^G = \k_1^G$
is the usual dual Hopf algebra of $\k G$.

\medskip
We are particularly concerned with \emph{cleft extensions} determined by
a pair of finite groups $F, G$.\ We assume that there is a right action $\triangleleft$  of $F$ on $G$ as
\emph{automorphisms} of $G$.\
The right $F$-action induces a natural left $\k F$-action on $\k^G$, making $\k^G$ a left $\k F$-module algebra.\ If
we consider $\k F$ as a trivial $\k^G$-comodule (i.e., $G$ acts trivially on $\k F$), then $(\k F, \k^G)$ is a
\emph{Singer pair}.\ Throughout this paper, only these special kinds of Singer pairs will be considered.

\medskip
A \emph{cleft object} of $\k^G_\w$ (or simply $G$) consists of a triple $c=(F, \g, \theta)$ where $c_0=F$ is a group
with a right action $\triangleleft$  on $G$ as automorphisms, and $c_1=\g \in C^2(G, (\k^F)\x), c_2 =\theta \in
C^2(F,  (\k^G)\x)$ are normalized $2$-cochains.\ They are required to satisfy the following conditions:
\begin{equation}\label{eq:2cocycle1}
 \theta_{g \triangleleft x}(y,z) \theta_g ( x, y z) =  \theta_g (x y,z) \theta_g ( x, y ),
 \end{equation}
 \begin{equation} \label{eq:2cocycle12}
 \g_x(gh, k) \g_x(g, h) \w(g \triangleleft x,h \triangleleft x ,k \triangleleft x)  = \g_x(h, k) \g_x(g, hk)
 \w(g,h,k),
  \end{equation}
   \begin{equation}\label{eq:2cocycle3}
   \frac{\g_{xy}(g, h)} {\g_x(g, h) \g_y(g \triangleleft x, h \triangleleft x)}=  \frac{\theta_g( x, y) \theta_h(x,
   y)}{\theta_{gh}(x, y)}
\end{equation}
where $\theta_g(x,y) := \theta(x,y)(g)$, $\g_x(g,h):=\g(g,h)(x)$ for $x,y \in F$ and $g,h \in G$.

\medskip
 Associated to a cleft object $c$ of $G$ is a quasi-Hopf algebra
 \begin{eqnarray}\label{cleftqH}
 H=\k^G_\w \#_{c}\, \k F
 \end{eqnarray}
 with underlying linear space $\k^G \o \k F$; the ingredients necessary to define the quasi-Hopf algebra
 structure are as follows:
 $$
e_g x \cdot e_h y =\delta_{g \triangleleft x, h}\,\theta_g(x,y)\, e_g xy, \quad 1_H=\sum_{g \in G} e_g , \,
$$
$$
\Delta(e_g x) = \sum_{ab=g} \g_x(a,b) e_a x \o e_b x, \quad \e(e_g x) =\delta_{g,1}\,,
$$
$$
S(e_g x) = \theta_{g\inv}(x, x\inv)\inv \g_x(g, g\inv)\inv e_{g\inv \triangleleft x} x\inv,
$$
$$
\a=1_H, \quad \b = \sum_{g\in G}\w(g, g\inv, g) e_g,
$$
where  $e_g x \equiv e_g \o x$ and $e_g \equiv e_g \o 1_F$.\ The Drinfel'd associator $\phi$ is again given by
(\ref{eq:associator}).\ This quasi-Hopf algebra is also called the \emph{cleft extension} of $kF$ by $\k^G_\w$ (cf.\ \cite{Masu02}).\ The proof that (\ref{cleftqH}) is indeed a quasi-Hopf algebra when equipped with these structures is
rather routine, and is similar to that of the \emph{twisted quantum double} $D^{\omega}(G)$, which is the case
when $F=G$ and the action on $G$ is conjugation (\cite{DPR}, \cite{K}).\ We shall return to this
example in due course.\ Note that these cleft extensions admit the canonical morphisms of quasi-Hopf algebras
\begin{eqnarray}\label{CGwobj}
 \k^G_\w \stackrel{i}{\to}  \k^G_\w \#_{c}\, \k F \stackrel{p}{ \to} \k F
\end{eqnarray}
where
$$
i(e_g) =e_g, \quad p(e_g x) = \delta_{g,1} x.
$$

\bigskip
Introduce the category $\Cleft(\k^G_\w)$ whose objects are the  cleft objects of $\k_\w^G$; a morphism from $c =(F,
\g, \theta)$ to
$c' =(F', \ol \g, \ol \theta)$ is a pair $(f_1, f_2)$ of quasi-bialgebra  homomorphisms satisfying that
\begin{enumerate}
  \item[(i)] $f_2$ preserves the actions on $G$, i.e. $g \triangleleft x = g \triangleleft f_2(x)$, and
  \item[(ii)] the diagram
  $$
  \xymatrix{
  \k^G_\w \ar[r]^-{i} \ar[d]^-{\id} &  \k^G_\w \#_{c}\, \k F \ar[r]^-{p} \ar[d]^-{f_1} & \k F \ar[d]^-{f_2} \\
   \k^G_\w \ar[r]^-{i} &  \k^G_\w \#_{c'} \k F' \ar[r]^-{p'} & \k F'
  }
  $$
  commutes.
\end{enumerate}
It is worth noting that $\Cleft(\k^G_\w)$ is essentially the category of cleft extensions of group algebras by
$\k_\w^G$.
\begin{remark}\label{r:grading}
The quasi-Hopf algebra $\k^G_\w \#_{c}\, \k F$ also admits a natural $F$-grading which makes it an $F$-graded algebra.\ This $F$-graded structure can be described in terms of  the $\k F$-comodule via the structure map
$\rho_c = (\id \o p)\Delta$\,.\ A morphism $(f_1, f_2): c \to c'$ in $\Cleft(\k_\w^G)$ induces the right $\k
F'$-comodule structure $\rho'_c = (\id \o f_2) \rho_c$ on $\k^G_\w \#_{c}\, \k F$, and $f_1:  \k^G_\w \#_{c}\, \k F
\to \k^G_\w \#_{c'}\, \k F'$ is then a right $\k F'$-comodule map.\ In the language of group-grading, $f_2$ induces an
$F'$-grading on $\k^G_\w \#_{c}\, \k F$ and $f_1$ is an $F'$-graded linear map.\ Since $f_1$ is an algebra map and preserves $F'$-grading, $f_1(e_g x) = \chi_x(g) e_g \ol x$ for some scalar
$\chi_x(g)$, where $\ol x = f_2(x) \in F'$ for $x \in F$. 
\end{remark}

\begin{remark}
  In general, a quasi-bialgebra  homomorphism between two quasi-Hopf algebras is \emph{not} a quasi-Hopf algebra
  homomorphism.\ However, if $(f_1, f_2)$ is a morphism in $\Cleft(\k_\w^G)$, then both $f_1$ and $f_2$ are
  quasi-Hopf algebra homomorphisms.\ We leave this observation as an exercise to readers (cf.\ \eqref{eq:eq1} and
  \eqref{eq:eq2} in the proof of Theorem \ref{thmcext} below).
\end{remark}

In $\Cleft(\k_\w^G)$, there is a trivial object $\ul 1$ in which the group $F$ is trivial and $\theta, \g$ are both
identically $1$.\ This cleft object is indeed the trivial cleft extension of $\k_\w^G$: $\k^G_\w
\stackrel{\id}{\rightarrow} \k^G_\w\stackrel{\epsilon}{\rightarrow} \k$. It is straightforward to check that $\ul 1$
is an initial object of $\Cleft(\k_\w^G)$.

\medskip
 Suppose we are given a cleft object $c =(F, \g, \theta)$ and a quotient map $\pi_{\bar{F}} : F \to \ol F$ of $F$ which
 preserves  their actions on $G$.\ We ask the following question: is there a cleft object $\ol c = (\ol F, \ol \g,
 \ol\theta)$ of $\k_\w^G$ and a quasi-bialgebra  homomorphism $\pi : \k_\w^G \#_c\, \k F  \to \k_\w^G \#_{\ol c}\, \k
 \ol F$ such that
$(\pi, \pi_{\bar{F}}): c \to \ol c$ is a morphism of $\Cleft(\k_\w^G)$?\ Equivalently, the diagram
\begin{equation}\label{diagcleft}
  \xymatrix{
  \k^G_\w \ar[r]^-{i} \ar[d]^-{\id} &  \k^G_\w \#_{c}\, \k F \ar[r]^-{p} \ar[d]^-{\pi} & \k F \ar[d]^-{\pi_{\bar{F}}}
  \\
   \k^G_\w \ar[r]^-{i} &  \k^G_\w \#_{\ol c}\, \k \ol F  \ar[r]^-{\ol p} & \k \ol F}
 \end{equation}
commutes.\ Generally, one can expect the answer to this question to be `no'.\ In the following section, we will
provide a complete answer in an important special case.

\section{Central quotients}\label{CQs}
Throughout this section we assume $\k$ is a field of \emph{any} characteristic,  $c= (F, \g, \theta)$ an object
of $\Cleft(\k_\w^G)$ with the associated quasi-Hopf algebra monomorphism $i:\k_\w^G \to \k_\w^G \#_{c}\, \k F$ and
epimorphism $p:\k_\w^G \#_{c}\, \k F  \to \k F$.\ We use the same notation as before, and  write $H=\k_\w^G
\#_{c}\, \k F$.

\medskip
 We now suppose that $A\subseteq Z(F)$ is a \emph{central} subgroup of $F$ such that
 restriction of the $F$-action $\triangleleft$ on $G$ to $A$ is \emph{trivial}.\ Then the quotient group $\ol F =F/A$
 inherits the right action, giving rise to an induced Singer pair $(\k \ol F, \k^G)$.\ With this set-up,
 we will answer the question raised in the previous section about the existence of
 the diagram (\ref{diagcleft}).\ To explain the answer, we need some preparations.

 \begin{defn}\label{defn1}
   \enumeri{
   \item $0\not= u \in H$ is \emph{group-like} if $\Delta(u)=u \o u$.\ The sets of  group-like elements and central
       group-like elements of $H$ are denoted by $\G(H)$ and  $\G_0(H)$ respectively.
   \item $x \in F$ is called $\g$-\emph{trivial} if $\g_x\in B^2(G, \k\x)$ is a 2-coboundary.\ The set of
       $\g$-trivial elements is denoted by $F^\g$.
   \item $a \in F$ is
    \emph{$c$-central} if there is $t_a \in C^1(G, \k\x)$ such that
  \begin{equation}\label{eq:c0}
     \sum_{g \in G} t_a(g) e_g a \in \G_0(H)\,.
  \end{equation}
   The set of $c$-central elements is denoted by $Z_{c}(F)$.}
 \end{defn}

Let $\hat G = \Hom(G, \k^\times)$ be the  group of linear characters of $G$. The following lemma concerning the sets $F^\g$, $\G(H)$ and $\hat G$ is similar to an observation in \cite{MN2}. 
\begin{lem}\label{lemmamixgps}
The following statements concerning  $F^\g$ and $\G(H)$ hold.
\enumeri{
\item
$F^{\gamma}$ is a subgroup of $F$,  $\G(H)$ is a subgroup of the group of units in $H$, and
$p(\G(H)) = F^\g$. Moreover, for $x \in F^\g$ and $t_x \in C^1(G, \k\x)$, 
$$
\sum_{g \in G} t_x(g) e_g x \in \G(H) \text{ if, and only if, } \delta t_x = \g_x\,.
$$
\item The sequence of groups
    \begin{equation}\label{eq:extension1}
     1 \to \hat G \xrightarrow{i}  \G(H) \xrightarrow{p} F^\g \to 1
    \end{equation}
    is exact.
    The $2$-cocycle $\b \in Z^2(F^\g, \hat G)$ associated with the section $x \mapsto \sum_{g \in G} t_x(g) e_g x$ of $p$ in (\ref{eq:extension1}) is given by
\begin{equation}\label{eq:2-cocycle}
\b(x,y)(g) =\frac{t_x(g)t_y(g\triangleleft x)}{t_{xy}(g)}\theta_g(x,y)\ \ (x,y \in F^\g, g \in G).
\end{equation}
}
\end{lem}
\begin{proof}
The proofs of (i) and (ii) are similar to Lemma 3.3 in \cite{MN2}.
\end{proof}
\begin{remark}
 \eqref{eq:extension1} is a \emph{central extension} if $F$ acts trivially on $\hat{G}$, but in general it is
 \emph{not} a central extension.
 \end{remark}
 \begin{remark}\label{r:c-central} If $a \in Z_c(F)$, then a central  group-like element $\sum_{g \in G} t_a(g) e_g a \in \G_0(H)$
 will be mapped to the central element $a$ in $\k F$ under $p$.\ Therefore, by Lemma \ref{lemmamixgps}, we always have $Z_{c}(F) \subseteq Z(F)\cap F^\g$.\
 By direct computation,  the condition \eqref{eq:c0} for $a \in Z_{c}(F)$ is equivalent to the conditions:\
   $$
 \delta t_a = \g_a, \, t_a(g) \theta_g(a,y) = t_a(g \triangleleft y) \theta_g(y,a)\ \text{and}\  g \triangleleft a=g
   \ \ (g \in G, y \in F).
   $$
   In particular, $\theta_g(a,b) = \theta_g(b,a)$ for all $a, b \in Z_c(F)$.
\end{remark} 
 
By Lemma \ref{lemmamixgps}, we can parameterize the elements $u=u(\chi, x) \in \G(H)$ by $(\chi, x) \in \hat G
\times F^\g$.\ More precisely, for a fixed family of 1-cochains $\{t_x\}_{x \in F^\g}$ satisfying $\delta t_x = \g_x$,  every element $u \in \G(H)$ is uniquely determined by a pair  $(\chi, x) \in \hat G
\times F^\g$ given by
$$
u=u(\chi, x) = \sum_{g \in G} \chi(g) t_x(g) e_g x.
$$
 Note that a choice of such a family of 1-cochains $\{t_x\}_{x \in F^\g}$ satisfying $\delta t_x = \g_x$ is equivalent to a section of $p$ in \eqref{eq:extension1}.\ With this convention we
have $i(\chi) = u(\chi, 1)$ and $p(u(\chi, x)) =x$ for all $\chi \in \hat{G}$ and $x \in F^\g$.
\begin{lem}\label{l:subext1}
The set $Z_{c}(F)$ of $c$-central elements is a subgroup of $Z(F)$,\ and it acts trivially on $G$.\ Moreover,
 $\G_0(H)$ is a central extension of $Z_c(F)$ by $\hat{G}^F$ via the exact sequence:
     \begin{equation}\label{eq:extension0}
     1 \to \hat{G}^F \xrightarrow{i}  \G_0(H) \xrightarrow{p} Z_{c}(F) \to 1\,,
    \end{equation}
    where $\hat{G}^F$ is the group of $F$-invariant linear characters of $G$.

    If we choose $t_x$  such that $u(1, x) \in \G_0(H)$ whenever $x \in Z_{c}(F)$, then the formula
    (\ref{eq:2-cocycle}) for $\b(x,y)$ defines a 2-cocycle for the exact sequence (\ref{eq:extension0}).
\end{lem}
\begin{proof}
By Lemma \ref{lemmamixgps} and the preceding paragraph,  $u(\chi,x) \in \G_0(H)$ for some $\chi \in \hat G$ if, and
only if,  $x \in Z_c(F)$.\ In particular, $p(\G_0(H))=Z_c(F)$.\ It follows from Remark \ref{r:c-central} that $Z_c(F)$ is a subgroup of
$F^\g \cap Z(F)$
 and $Z_{c}(F)$ acts trivially on $G$.\ By Remark \ref{r:c-central} again, $u(\chi, 1) \in \G_0(H)$ is equivalent to
  $$
  \chi(g) t_1(g) \theta_g(1,y) = \chi(g \triangleleft y) t_1(g \triangleleft y) \theta_g(y,1) \text{ for all } g \in G,
  y \in F.
   $$
 In particular, $\hat{G}^F = \ker p|_{\G_0(H)}$, and this establishes the exact sequence \eqref{eq:extension0}. If
 $t_x$ is chosen such that $u(1,x) \in \G_0(H)$ whenever $x \in Z_c(F)$, the second statement follows immediately
 from Lemma \ref{lemmamixgps} (ii) and the commutative diagram:
 $$
  \xymatrix{
  1 \ar[r] & \hat G \ar[r]^-{i} & \G(H) \ar[r]^-{p} & F^\g \ar[r] & 1 \\
  1 \ar[r] & \hat G^F\ar[u]_-{incl} \ar[r]^-{i} & \G_0(H)\ar[u]_-{incl} \ar[r]^-{p} & Z_c(F)\ar[u]_-{incl} \ar[r] &
  1
  }\quad\raisebox{-42pt}{.  \qedhere}
$$
\end{proof}

\begin{thm}\label{thmcext} Let the notation be as before, with $A\subseteq Z(F)$ a subgroup acting trivially on $G$,
and with the right action of $\ol F = F/A$ on $G$ inherited from that of $F$.\ Then the following statements are
equivalent:
 \enumeri{
 \item
 There exist a cleft object $\ol c=(\ol F, \ol \g, \ol \theta)$ of $\k_\w^G$ and a quasi-bialgebra map $\pi: \k^G_\w
 \#_{c}\, \k F \to \k^G_\w \#_{\ol c}\, \k \ol F$ such that the diagram
  \begin{equation}\label{eq:cd1}
  \xymatrix{
  \k^G_\w \ar[r]^-{i} \ar[d]^-{\id} &  \k^G_\w \#_{c}\, \k F \ar[r]^-{p} \ar[d]^-{\pi} & \k F \ar[d]^-{\pi_{\ol F}}
  \\
   \k^G_\w \ar[r]^-{i} &  \k^G_\w \#_{\ol c}\, \k \ol F \ar[r]^-{p'} & \k \ol F
   }
 \end{equation}
commutes.
 \item $A \subseteq Z_{c}(F)$ and the subextension
 $$
 1 \to \hat G^F \xrightarrow{i} p|_{\G_0(H)}\inv(A) \xrightarrow{p} A \to 1
 $$
 of \eqref{eq:extension0} splits.
 \item $A \subseteq Z_{c}(F)$ and there exist $\{t_a\}_{a \in A}$ in $C^1(G, \k\x)$ and $\{\tau_g\}_{g \in G}$ in
     $C^1(A, \k\x)$ such that $\delta t_a = \g_a$, $\delta \tau_g = \theta_g |_A$ and
     $$
    s_a(g)= t_a(g) \tau_g(a)
     $$
     defines a $F$-invariant linear character on $G$ for all $a \in A$.
 }
  \end{thm}
\begin{proof} ((i) $\Rightarrow$ (ii)) Suppose there exist a cleft object $\ol c= (\ol F, \ol \g, \ol \theta)$ of
$\k_\w^G$ and
a quasi-bialgebra map $\pi: \k^G_\w \#_{c}\, \k F \to \k^G_\w \#_{\ol c}\, \k \ol F$ such that the diagram
\eqref{eq:cd1}
 commutes.\ Then $\pi(e_g) = e_g$ for all $g \in G$.\ Since $\pi$ is an algebra map, $\pi(e_g x) = \sum_{\ol y \in F}
 \chi_{x}(g, \ol y) e_g \ol y$ for some scalars $\chi_{x}(g, \ol y)$. Here, we simply write $\ol y$ for $\pi_{\ol
 F}(y)$.

 By Remark \ref{r:grading}, $\pi$ is a $\ol F$-graded linear map and so we have $\pi(e_g x) = \chi_{x}(g, \ol x) e_g
 \ol x$. Therefore, we simply denote $\chi_x(g)$ for $\chi_{x}(g, \ol x)$. In particular, $\chi_1=1$ and $\chi_x(1)
 =1$ by the commutativity of \eqref{eq:cd1}.
 Moreover, we find
 \begin{eqnarray}
   \g_x(g,h) \chi_{x}(g) \chi_{x}(h) & = & \ol \g_{\ol x}(g,h) \chi_{x}(gh),  \label{eq:eq1}\\
   \ol \theta_g(\ol x, \ol y) \chi_{x}(g) \chi_{y}(g \triangleleft x) & = & \theta_g(x, y)
   \chi_{xy}(g)\label{eq:eq2}
 \end{eqnarray}
  for all $x,y \in F$ and $g,h \in G$. An immediate consequence of these equations is that $\chi_x \in C^1(G,
  \k^\times)$ for $x \in F$.

  For $a \in A$, $\ol \theta_g(\ol a, \ol y) = \ol \g_{\ol a}(g,h)=1$. Then, \eqref{eq:eq1} and \eqref{eq:eq2} imply
  \begin{equation} \label{eq:eq3}
    \g_a = \delta \chi_a\inv, \quad 1=   \frac{\chi_{ay}(g)} {\chi_a(g) \chi_{y}(g)} \theta_g(a, y)=
  \frac{\chi_{ya}(g)} {\chi_{y}(g) \chi_a(g \triangleleft y)}\theta_g(y,a)
  \end{equation}
  for all $a \in A$, $g \in G$ and $y \in F$.\ These equalities in turn yield
  $$
  \sum_{g \in G} \chi_a\inv (g) e_g a \in \G_0(H)
  $$
  for all $a \in A$.\ Therefore, $A \le Z_{c} (F)$.

  In particular, $A \subseteq F^\g$.\ If we choose $t_a =\chi_a\inv$ for all $a \in A$, then the restriction of the
  2-cocycle $\b$, given in \eqref{eq:2-cocycle}, on $A$ is constant function 1.\ Therefore,  the subextension
  $$
 1 \to \hat G^F \xrightarrow{i} p|_{\G_0(H)}\inv(A) \xrightarrow{p} A \to 1
 $$
 of \eqref{eq:extension0} splits.\\ \\
  ((ii) $\Rightarrow$ (i) and (iii)) Assume $A \subseteq Z_c (F)$ and the restriction of $\b$ on $A$ is a
  coboundary.\ By Remark \ref{r:c-central}, we let $t_a \in C^1(G,\k\x)$ such that $\delta t_a =\g_a$ and
  \begin{equation}\label{eq:c1}
    t_a(g)\theta_g(a,y) = t_a(g \triangleleft y)\theta_g(y, a)
  \end{equation}
  for all $a \in A$, $y \in F$ and $g \in G$.\ In particular,
  $$
\sum_{g \in G} t_a(g) e_g a \in \G_0(H)
  $$
  for all $a \in A$.\
  By Lemma \ref{l:subext1},  $\b(a,b) \in \hat G^F$ for all $a, b \in A$.\ Suppose $\nu=\{\nu_a\mid a \in A\}$ is a family in $\hat
  G^F$ such that $\b(a,b) = \nu_a \nu_b \nu_{ab}\inv$ for all $a, b \in A$.

  Let $\ol r: \ol F \to F$ be a section of $\pi_{\ol F}$ such that $\ol r(\ol 1)=1$.\ For $x \in F$, we set $r(x)=\ol
  r(\ol x)$ and
   \begin{equation} \label{eq:isomorphism-scalars}
     \chi_x(g) = \frac{\nu_a(g)}{t_a(g) \theta_g(a, r(x))}
   \end{equation}
   for all $g \in G$, where $a = x r(x)\inv$.\ It is easy to see that $\chi_1 =1$ and $\chi_x$ is a normalized
   1-cochain of $G$.\ Note that for $b \in A$, $\theta_g(a,b) = \theta_g(b,a)$, so we have
   \begin{equation}\label{eq:c2}
    \frac{\chi_{bx}(g)}{\chi_{b}(g) \chi_{x}(g)} = \frac{\nu_{ab}(g)}{t_{ab}(g) \theta_g(ab, r(x))}
   \frac{t_b(g)} {\nu_b(g)}  \frac{t_a(g) \theta_g(a, r(x))}{\nu_a(g)}=\theta_g(b,x)\inv,
   \end{equation}
   \begin{equation}\label{eq:c3}
  \chi_b(g \triangleleft x)\theta_g(b,x) = \chi_b(g) \theta_g(x, b) \quad\text{and}\quad \delta \chi_b\inv  =\g_b\,.
   \end{equation}

    Let $\tau_g(a) = \chi_a(g)$ for all $a \in A$ and $g \in G$.\ Equation \eqref{eq:c2} implies that $\delta \tau_g
    =\theta_g|_A$ and
   $$
   \nu_a(g) =t_a(g)\tau_g(a)\,,
   $$
   and this proves (iii).

   Define the maps $\ol \g\in  C^2(G, (\k^{\ol F})\x)$ and $\ol \theta \in C^2(\ol F, (\k^G)\x)$ as follows:
   \begin{eqnarray}
   \ol \g_{\ol x}(g,h) & = & \frac{\chi_x(g)\chi_x(h)}{\chi_x(g h)} \g_x(g,h)\,,\label{eq:quo-gamma}\\
   \ol \theta_g(\ol x , \ol y) & = & \frac{\chi_{xy}(g)}{\chi_x(g)\chi_y(g \triangleleft x)} \theta_g(x,y)\,.\label{eq:quo-theta}
   \end{eqnarray}
   We need to show that these functions are well-defined.\ Let $b \in  A$, $x,y \in F$ and $g, h \in G$. By
   \eqref{eq:2cocycle3},  \eqref{eq:c2} and \eqref{eq:c3}, we find
   $$
   \frac{\chi_{bx}(g)\chi_{bx}(h)}{\chi_{bx}(g h)} \g_{bx}(g,h) = \frac{\chi_{x}(g)\chi_{x}(h)}{\chi_{x}(g h)}
   \g_{x}(g,h)\,,
   $$
   and this proves $\ol \g$ is well-defined.\ To show that $\ol \theta$ is also well-defined, it suffices to prove
   $$
   \frac{\chi_{bxy}(g)}{\chi_{bx}(g)\chi_y(g \triangleleft bx)} \theta_g(bx,y) =
   \frac{\chi_{xy}(g)}{\chi_{x}(g)\chi_y(g \triangleleft x)} \theta_g(x,y) =
    \frac{\chi_{xby}(g)}{\chi_{x}(g)\chi_{by}(g \triangleleft x)} \theta_g(x,by)
   $$
   for all $b \in A$, $x,y \in F$ and $g, h \in G$.
   However, the first equality follows from \eqref{eq:c2} and \eqref{eq:2cocycle1} while the second equality is a
   consequence of \eqref{eq:2cocycle1}, \eqref{eq:c2} and \eqref{eq:c3}.

\medskip
It is straightforward to verify that $\ol c=(\ol F, \ol \g, \ol \theta)$ defines cleft object of $\k_\w^G$ and $\pi:
\k^G_\w \#_{c}\,\k F \to \k^G_\w \#_{\ol c}\, \k \ol F, e_g x \mapsto \chi_x(g) e_g \ol x$ defines a quasi-bialgebra
homomorphism  which makes the diagram \eqref{eq:cd1} commute. We leave the routine details to the reader.\\\\
((iii) $\Rightarrow$ (ii)) Since $s_a(g) = t_a(g) \t_g(a)$ defines a $F$-invariant linear character of $G$ for each
$a$, then $\nu(a) = s_a$ defines a 1-cochain in $C^1(A, \hat{G}^F)$ and
$$
\delta \nu = \b|_A
$$
where $\b$ is the 2-cocycle given in \eqref{eq:2-cocycle}. In particular, $\b|_A$ is a coboundary.
\end{proof}
\begin{remark}\label{r:non-uniqueness} 
  Given $A \subseteq Z_c(A)$ satisfying  condition (ii) of the preceding theorem, and $\{t_a\}_{a \in A}$ a fixed family of cohains in $C^1(G, \k^\times)$ such that $\sum_{g \in G} t_a(g) e_g a \in \G_0(H)$ for $a \in A$. The set $\SS(A)$ of group homomorphism sections of $p : p\inv(A) \to A$ is in one-to-one corresponence with $\BB(A) = \{\nu \in C^1(A, \hat G^F)\mid \delta \nu = \b \text{ on }A\}$. For $\nu \in \BB(A)$, it is easy to see that 
  $$
  \tilde{p}_\nu (a) = \sum_{g \in G} \frac{t_a(g)}{\nu(a)(g)}\, e_g a \quad (a \in A)
  $$
  defines a group homomorphism in $\SS(A)$. Conversely,  if $\tilde p' \in \SS(A)$,   then there exists a group homomorphism $f: A \to \hat G^F$ such that 
$i (f(a))\tilde p'(a) =  \tilde p(a)$ for all $a \in A$. In particular, if  $\tilde p'(a) = 
\sum_{g \in G} t'_a(g) e_g a$ for $a \in A$, then 
$$
t'_a = \frac{t_a}{\nu(a) f(a)}
$$
and $\nu'=\nu f \in \SS(A)$. Therefore, $\tilde p' = \tilde p_{\nu'}$.
 
 The cleft object $\ol c=(F/A, \ol g,
  \ol \theta)$  and morphism $\pi$ contructed in the proof of Theorem \ref{thmcext} are \emph{not}  unique.\ The definition of $\chi_x(g)$ is determined by
  the choice of the section $\ol r: \ol F\to F$ of $\pi_{\ol F}$  and   $\nu \in \BB(A)$.
  If $\nu' \in \BB(A)$, then $\nu'=\nu f$ for some group homomorphism   $f: A \to \hat G^F$.
  Thus, the corresponding 
  $$
  \chi'_x(g) = f(xr(x)\inv)(g)\chi_x(g)\,.\
  $$
  This implies $\ol c'= (F/A, \ol\g', \ol \theta')$ where $\ol \g' = \ol \g$ but
  $$
  \ol\theta'_g(\ol x, \ol y) =\frac{\ol \theta_g(\ol x, \ol y)}{f(r(x)r(y)r(xy)\inv)(g)} \, .
  $$
  Therefore, $\ol c$ as well as $\pi$ can be altered by the choice of any group homomorphism $f:A \to \hat{G}^F$ for
  a given section $\ol r : \ol F \to F$ of $\pi_{\ol F}$.
\end{remark}

\section{Cleft objects for the twisted quantum double $D^\w(G)$}\label{DG}
Consider the right action of a finite group $F=G$ on itself by conjugation with $\w \in Z^3(G, \k\x)$ a normalized
3-cocycle.\ We will write $x^g=g\inv xg$.\ There is a \emph{natural} cleft object $c_\w=(G, \g, \theta)$  of
$\k_\w^G$
 given by
\begin{equation}\label{eq:w_cleftdatum}
  \g_g(x,y)=\frac{\w(x,y, g)\w(g, x^g, y^g)}{\w(x, g, y^g)}, \quad \theta_g(x,y) = \frac{\w(g, x,y)\w(x, y ,
  g^{xy})}{\w(x, g^x, y)}\,.
\end{equation}
Note that $\g_z =\theta_z$ for any $z \in Z(G)$.\
The associated quasi-Hopf algebra $D_\k^\w(G)=\k^G_\w \#_{c_\w} \k G$ of this natural cleft object $c_\w$ is the
\emph{twisted quantum double}
 of $G$ [DPR]. From now on, we simply abbreviate $D_\k^\w(G)$ as $D^\w(G)$ when $\k$ is the field of complex numbers
 $\BC$.\

 For the cleft object $c_\w$, we can characterize the $c_\w$-central elements in the following result
 (cf.\ Lemma \ref{l:subext1}).
\begin{prop}\label{p:central1}
  The $c_\w$-center $Z_{c_\w}(G)$ is given by
  $$
  Z_{c_\w}(G) = Z(G)\cap G^\g.
  $$
 The group $\G_0(D^\w(G))$ of central group-like elements of $D^\w(G)$ is the middle term of the 
 short exact sequence
  $$
  1 \to \hat{G} \xrightarrow{i} \G_0(D^\w(G)) \xrightarrow{p} Z(G) \cap G^\g \to 1\,.
  $$
  In addition, if $H^2(G, \k\x)$ is trivial, then $Z(G) = Z_{c_\w}(G)$.
\end{prop}
\begin{proof}
The inclusion $Z_{c_\w}(G) \subseteq Z(G)\cap G^\g$ follows directly from  Remark \ref{r:c-central}.\ Suppose $z \in
Z(G) \cap G^\g$ and choose $t_z \in C^1(G, \k\x)$ so that $\delta t_z =\g_z$.\ Since $z \in Z(G)$, $\theta_z=\g_z$
and so $\theta_z =\delta t_z$.\ This implies
$$
\frac{\theta_z(y, g^y)}{\theta_z(g, y)} = \frac{t_z(g^y)}{t_z (g)}\ \ (g, y \in G).
$$
 It follows directly from the definition \eqref{eq:w_cleftdatum} of $\theta$ that
$$
\frac{\theta_g(z,y)}{\theta_g(y,z)}=\frac{\theta_z(y, g^y)}{\theta_z(g, y)}\,.
$$
Thus, we have
$$
t_z(g) \theta_g(z,y) = t_z(g^y) \theta_g(y,z)\ \ (g, y \in G).
$$
 It follows from Remark \ref{r:c-central} that $z \in Z_{c_\w}(G)$.\ Since $\hat G = \hat G^G$, the exact sequence
 follows from Lemma \ref{l:subext1}.

\medskip
Finally, if  $H^2(G, \k\x)$ is trivial and $z \in Z(G)$,  then $\g_z \in B^2(G, \k\x)$ and therefore
$z \in G^{\gamma}$.\ The equality
$Z(G)=Z(G)\cap G^\g = Z_{c_\w}(G)$ follows.
\end{proof}
\begin{defn}
  In light of Theorem \ref{thmcext}, for the canonical cleft object $c_\w = (G, \g, \theta)$ of $\k_\w^G$, a subgroup  $A \subseteq
  Z(G)$ is called $\w$-\emph{admissible} if $A$ satisfies one of the conditions in Theorem \ref{thmcext}.\ The
  quasi-Hopf algebra $\k_\w^G \#_{\ol c_\w} \k \ol G$ of an associated cleft object $\ol c_\w =(\ol G=G/A, \ol \g, \ol\theta)$ is simply denoted by $D_{r,\tilde p}^\w(G, A)$.\ It depends on the choice of  a section $r$ of $\pi_{\ol G}: G \to \ol G$ and a group homomorphism section $\tilde p:A \to \G_0(D^\w(G))$ of $p: p\inv(A) \to A$
 (cf.\ Remark \ref{r:non-uniqueness}). We  drop the subscripts
  $r,\tilde p$ if there is no ambiguity.
\end{defn}

\begin{remark}\label{r:difference}
 The quasi-Hopf algebra $D^\w(G, N)$ constructed in \cite{GM}, where $N \unlhd G$ and $\w$ is an inflation of a 3-cocycle of $G/N$, is a completely different construction from the one presented with the same notation in the preceding definition.\ Both are attempts to generalized the twisted quantum double construction by taking subgroups into account.
\end{remark}

\begin{example}
  Let $Q$ be the quaternion group of order $8$ and $A=Z(Q)$.\ Since $H^2(Q, \BC\x)=0$, $A$ is $c_{\w}$-central for all
  $\w \in Z^3(Q, \BC\x)$.\ Since $\hat{Q} \cong \BZ_2 \times \BZ_2$, the associated 2-cocycle $\b$ of the extension
  $$
  1 \to \hat{Q} \to \G_0(D^\w(Q)) \to Z(Q) \to 1
  $$
  has order 1 or 2.\ Thus, if $\w$ is a square of another 3-cocycle, $\b=1$ and so $A$ is $\w$-admissible.\ In fact,
  $A$ is $\w$-admissible for all 3-cocycles of $Q$ but the proof is a bit more complicated.
\end{example}

\section{Simple currents and $\w$-admissible subgroups}\label{Smodtc}
For simplicity,  we will mainly work over the base field $\BC$ for the remaining discussion.\ Again, we assume that $G$ is a finite group and $\w \in Z^3(G, \BC\x)$ a normalized 3-cocycle.\ An isomorphism class of a 1-dimensional $D^\w(G)$-module is also called a \emph{simple current} of $D^\w(G)$.\ The set $\SC(G, \w)$ of all simple currents of $D^\w(G)$ forms a finite group with respect to tensor product of $D^\w(G)$-modules.\ The inverse of a simple current $V$ is the left dual $D^\w(G)$-module $V^*$.\  $\SC(G, \w)$ is also called the group of invertible objects of $\Rep(D^\w(G))$ in some articles.\ Since the category $\Rep(D^\w(G))$ of finite-dimensional $D^\w(G)$-modules is a braided monoidal category, $\SC(G, \w)$ is  
\emph{abelian}.

\medskip
Recall that each simple module $V(K, t)$ of $D^\w(G)$ is characterized by a  conjugacy class $K$ of $G$ and a character $t$ of the twisted group algebra $\BC^{\theta_{g_K}}(C_G(g_K))$, where $g_K$ is a fixed element of $K$ and $C_G(g_K)$ is the centralizer of $g_K$ in $G$.\ The degree of the module $V(K, t)$ is equal to $|K|t(1)$ (\cite{DPR}, \cite{KMM}).

\medskip
Suppose $V(K, t)$ is 1-dimensional.\ Then $K=\{z\}$ for some $z \in Z(G)$ and $t$ is a 1-dimensional character of $\BC^{\theta_z}(G)$.\ Thus, for $g,h \in G$, we have
\begin{equation}\label{eq:theta_chi}
\theta_z(g,h)t(\widetilde{gh}) = t(\tilde{g})t(\tilde{h}),
\end{equation}
where $\tilde g$ denotes $g$ regarded as an element of $\BC^{\theta_z}(G)$.\
Defining $t(g) =t(\tilde g)$ for $g \in G$, we see that $\theta_z = \g_z =\delta t$ is a 2-coboundary of $G$.\ Hence $z \in G^\g \cap Z(G)$. By Proposition \ref{p:central1}, $z \in Z_{c_\w}(G)$.\  Conversely, if $z \in Z_{c_\w}(G)$, then there exists $t_z \in C^1(G, \BC\x)$ such that $\delta t_z = \g_z$.\ Then $V(z, t_z)$ is a 1-dimensional  $D^\w(G)$-module.\ Thus we have proved

\begin{lem} Let $K$ be a conjugacy class of $G$, $g_K$ a fixed element of $K$ and $t$ an irreducible character of  $\BC^{\theta_{g_K}}(C_G(g_K))$.\ Then $V(K,t)$ is a simple current of $D^\w(G)$ if, and only if, $K=\{z\}$ for some $z \in Z_{c_\w}(G)$ and $\delta t = \theta_z$. \qed
\end{lem}

For simplicity, we denote the simple current $V(\{z\}, t)$ by $V(z, t)$.\ By \cite{DPR} or \cite{KMM} the character $\xi_{z, t}$ of $V(z, t)$ is given by
\begin{equation} \label{eq:character}
  \xi_{z, t}(e_g x) = \delta_{g,z}t(x)\,.
\end{equation}

Fix a family of normalized 1-cochains $\{t_z\}_{z \in Z_{c_\w}(G)}$ such that $\delta t_z =\g_z$.\ Then for any simple current $V(z, t)$ of $D^\w(G)$,  $t$ is a normalized 1-cochain of $G$ satisfying $\delta t = \theta_z$.\ Thus, $t = t_z \chi$ for some $\chi \in \hat{G}$.\ Therefore, 
$$
\SC(G, \w)=\{V(z, t_z\chi)\mid z \in Z_{c_\w}(G) \text{ and } \chi \in \hat{G}\}\,.
$$

\medskip
Suppose $V(z', t_{z'}\chi')$ is another simple current of $D^\w(G)$.\ Note that
\begin{equation}
  \g_x(z,z') = \theta_x(z,z') \text{ and } \g_z(x,y) = \theta_z(x,y)
\end{equation}
for all $z, z' \in Z(G)$ and $x,y \in G$.\ By considering the action of $e_g x$, we find
\begin{equation}\label{eq:multiplication}
  V(z, t_z\chi) \o V(z', t_{z'}\chi') = V(z z', t_{zz'}\b(z,z') \chi \chi')
\end{equation}
where $\b$ is given by \eqref{eq:2-cocycle}.\ Therefore,  we have an exact sequence
$$
\xymatrix{
1\ar[r] &\ar[r] \hat G \ar[r]^-{i} &\SC(G, \w) \ar[r]^-{p} & Z_{c_\w}(G) \ar[r] & 1
}
$$
of abelian groups, where $i: \chi \mapsto  V(1, \chi)$ and $p: V(z, t_z \chi) \mapsto z$.\ With the same fixed family $\{t_z\}_{z \in Z_{c_\w}(G)}$ of 1-cochains, $u(\chi, z) = \sum_{g \in G} t_z(g) e_g z$ ($z \in Z_{c_\w}(G)$, $\chi \in \hat{G}$) are all the central group-like elements of $D^\w(G)$.\ By Lemma \ref{l:subext1}, the 2-cocycle associated with the extension
$$
\xymatrix{
1\ar[r] &\ar[r] \hat G \ar[r]^-{i} & \G_0(D^\w(G)) \ar[r]^-{p} & Z_{c_\w}(G) \ar[r] & 1
}
$$
is also $\b$, and so we have proved
\begin{prop}\label{p:iso-extensions}
Fix a family $\{t_z\}_{z \in Z_{c_\w}(G)}$  in $C^1(G, \BC\x)$ such that $\delta t_z =\theta_z$.\ Then the
map $\zeta: \G_0(D^\w(G)) \to \SC(G, \w)$, $u(\chi, z)  \mapsto V(z, t_z\chi)$ for $\chi \in \hat{G}$ and $z \in Z_{c_\w}(G)$,  defines an isomorphism of the following extensions:
$$
\xymatrix{
1\ar[r] &\ar[r] \hat G \ar[r]^-{i} &\SC(G, \w) \ar[r]^-{p} & Z_{c_\w}(G) \ar[r] & 1& \\
1\ar[r] & \ar[r] \hat G \ar[u]^-{\id} \ar[r]^-{i}&\G_0(D^\w(G))\ar[u]^-{\zeta} \ar[r]^-{p} & Z_{c_\w}(G)\ar[u]^-{\id} \ar[r] & 1.&\quad\qed
}
$$
\end{prop}
\begin{remark} The preceding proposition implies that these extensions depend only on the cohomology class of $\w$.\ In fact, if $\w$ and $\w'$ are cohomologous 3-cocycles of $G$, $Z_{c_\w} (G) = Z_{c_{\w'}}(G)$ but $\G(D^\w(G))$ and $\G(D^{\w'}(G))$ are not necessarily isomorphic. 
\end{remark}

In view of Proposition \ref{p:iso-extensions}, we will identify the group of simple currents $\SC(G,\w)$ with the group $\G_0(D^\w(G))$ of central group-like elements of $D^\w(G)$ under the map $\zeta$.\ In particular, we simply write the simple current $V(z, t_z \chi)$ as $u(\chi, z)$.

\medskip
The associativity constraint $\phi$ and the braiding $c$ of $\Rep(D^\w(G))$ define an Eilenberg-MacLane 3-cocycle $(\tilde\phi, d)$ of $\SC(G, \w)$ (\cite{EMac1, EMac2}) given by
\begin{multline}
  \tilde\phi\inv(u(\chi_1, z_1), u(\chi_2, z_2), u(\chi_3, z_3)) :=\\
\left( (u(\chi_1, z_1)\o u(\chi_2, z_2))\o u(\chi_3, z_3) \xrightarrow{\phi\cdot} u(\chi_1, z_1)\o u(\chi_2, z_2)\o u(\chi_3, z_3) \right)
\end{multline}
and
\begin{multline}
d(u(\chi_1, z_1)| u(\chi_2, z_2)):= c_{u(\chi_1, z_1), u(\chi_2, z_2)}=\\
\left(u(\chi_1, z_1)\o u(\chi_2, z_2) \xrightarrow{R\cdot} u(\chi_1, z_1)\o u(\chi_2, z_2) \xrightarrow{flip} u(\chi_2, z_2)\o u(\chi_1, z_1)\right),
\end{multline}
where $R=\sum_{g,h \in G} e_g \o e_h g$ is the universal $R$-matrix of $D^\w(G)$.\ By \eqref{eq:character}, one can compute directly that
\begin{eqnarray}
  \tilde\phi(u(\chi_1, z_1), u(\chi_2, z_2), u(\chi_3, z_3)) & = & \w(z_1, z_2, z_3)\,,\\
  d(u(\chi_1, z_1)| u(\chi_2, z_2)) & = & \chi_2(z_1) t_{z_2}(z_1)\,.
\end{eqnarray}

\medskip
The double braiding on $u(\chi_1, z_1)\o u(\chi_2, z_2)$ is then the scalar
$$
d(u(\chi_1, z_1)| u(\chi_2, z_2))\cdot d(u(\chi_2, z_2)| u(\chi_1, z_1)),
$$
which defines a symmetric bicharacter $(\cdot|\cdot)$ on $\SC(G, \w)$.\ Using \eqref{eq:character} to compute directly, we obtain
$$
(u(\chi_1, z_1)| u(\chi_2, z_2))=\chi_1(z_2) \chi_2(z_1) t_{z_2}(z_1)t_{z_1}(z_2)
$$
for all $u(\chi_1, z_1), u(\chi_2, z_2) \in \SC(G, \w)$.\ In general, $\SC(G, \w)$ is degenerate relative to this symmetric bicharacter $(\cdot|\cdot)$.\ However, there could be non-degenerate subgroups of $\SC(G,\w)$.

\begin{remark} \label{r:modularity}
It follows from \cite[Cor 7.11]{MugerII03} or \cite[Cor. 2.16]{Muger3} that a subgroup $A\subseteq \SC(G, \w)$ is non-degenerate if, and only if, the full subcategory $\AA$ of $\Rep(D^\w(G))$ generated by $A$ is a modular tensor category.
\end{remark}

We now assume $A$ is an $\w$-admissible subgroup of $G$.  Let $\nu$ be a normalized cochain in $C^1(A, \hat{G})$ such that $\b(a,b) =  \nu(a) \nu(b) \nu(ab)\inv$ for all $a, b \in A$.\ Therefore, by Remark \ref{r:non-uniqueness},
the assignment $\tilde p_\nu:a \mapsto u(\nu(a)\inv, a)$ defines a group monomorphism from $A$ to $\SC(G,\w)$ which is also a section of $p: p\inv(A) \to A$.\ Hence $A$ admits a bicharacter $(\cdot| \cdot)_\nu$ via the restriction of $(\cdot| \cdot)$ to $\tilde p_\nu(A)$.\ In particular,
\begin{equation}\label{eq:induced-bicharacter}
  (a| b)_\nu= (\tilde p_\nu(a)| \tilde p_\nu(b)) = \frac{t_b(a) t_a(b)}{\nu(b)(a)\nu(a)(b)}\,.
\end{equation}
Obviously, $(\cdot|\cdot)_\nu$ is non-degenerate if, and only if, $\tilde p_\nu(A)$ is a non-degenerate subgroup of $\SC(G,\w)$.\ On the other hand, $\nu$ also defines the quasi-Hopf algebra $D^\w(G, A)$ and a surjective quasi-Hopf algebra homomorphism $\pi_\nu: D^\w(G) \to D^\w(G, A)$.\ In particular, $\Rep(D^\w(G, A))$ is a tensor (full) subcategory of $\Rep(D^\w(G))$, so it inherits the braiding $c$ of $\Rep(D^\w(G))$.\ We can now state the main theorem in this section.

\begin{thm}\label{t:modular} 
Let $A$ be an $\w$-admissible subgroup of $G$, $\nu$  a normalized cochain in $C^1(A, \hat{G})$ , and $\tilde p_\nu:A \to \SC(G,\w)$ the associated group monomorphism. Then $$
c_{\tilde p_\nu(a), V}\circ c_{V, \tilde p_\nu(a)} = \id_{V \o \tilde p_\nu(a)}
$$
 for all $a \in A$ and irreducible $V \in \Rep(D^\w(G, A))$.\ Moreover, $\Rep(D^\w(G,A))$ is a modular tensor category if, and only if, the bicharacter $(\cdot|\cdot)_\nu$ on $A$ is non-degenerate.
\end{thm}
\begin{proof}
  Since a braiding $c_{U,V} : U \o V \to V \o U$ is a natural isomorphism and the regular representation $U$ of $D^\w(G,A)$ has every irreducible $V \in \Rep(D^\w(G,A))$ as a summand, it suffices to show that
  $$
  c_{\tilde p_\nu(a), U} \circ c_{U,\tilde p_\nu(a)} = \id_{U\o \tilde p_\nu(a)}
  $$
  for all $a \in A$.\ Let $\ol c_\w = (G/A=\ol G, \ol \theta, \ol \gamma)$ be an associated cleft object of $\BC_\w^G$ and $\pi_\nu: D^\w(G) \to D^\w(G,A)$ an epimorphism of quasi-Hopf algebras constructed in the proof of Theorem \ref{thmcext} using $\nu$.\ In particular, $\pi_\nu(e_g x) = \chi_x(g)\, e_g \ol x$ for all $g, x \in G$ where $\ol x$ denotes the coset $xA$ and the scalar $\chi_x(g)$ is given by
  \eqref{eq:isomorphism-scalars}.

\medskip
  Let $\1_{\tilde p_\nu(a)}$ denote a basis element of $\tilde p_\nu(a)=V(a, t_a \nu(a)\inv)$. Then, by \eqref{eq:character},
  $$
  e_g x \cdot  \1_{\tilde p_\nu(a)} = \delta_{g,a} \frac{t_a(x)}{\nu(a)(x)}\, \1_{\tilde p_\nu(a)}\,.
  $$
  Note that we can take $U = D^\w(G,A)$ as a $D^\w(G)$-module via $\pi_\nu$, and so
  $$
  e_g x \cdot e_h \ol y = \pi_\nu(e_g x) e_h \ol y = \delta_{g^x, h} \,\chi_x(g)\, \ol \theta_g (\ol x, \ol y) e_g \ol{xy}
  $$
  for all $g, h, x,y \in G$. Since the  $R$-matrix of $D^\w(G)$ is given by $R=\sum_{g, h \in G} e_g \o e_h g$, we have
  \begin{eqnarray*}
     c_{\tilde p_\nu(a), U} \circ c_{U,\tilde p_\nu(a)} (e_g \ol y \o \1_{\tilde p_\nu(a)}) & = &
   R^{21} R\cdot (e_g \ol y \o \1_{\tilde p_\nu(a)}) \\
   &=& \frac{t_a(g)}{\nu(a)(g)} R^{21}\cdot (e_g \ol y \o \1_{\tilde p_\nu(a)}) \\
   &=& \frac{t_a(g)}{\nu(a)(g)} \chi_a(g) \ol \theta_g(\ol a, \ol y) \, e_g \ol y \o \1_{\tilde p_\nu(a)} \\
   &=&  e_g \ol y \o \1_{\tilde p_\nu(a)} \\
  \end{eqnarray*}
  for all $a \in A$.\ This proves the first assertion.
  
  \medskip
  Let $\AA$ be the full subcategory of $\CC=\Rep(D^\w(G))$ generated by $\tilde p_\nu(A)$.\ The first assertion of the theorem implies  that $\Rep(D^\w(G,A))$ is a full subcategory of the centralizer $C_{\CC}(\AA)$ of $\AA$ in $\CC$.\ Since $\dim \AA = |A|$ and $\Rep(D^\w(G))$ is a modular tensor category, by \cite[Thm. 3.2]{Muger3}, 
  $$
  \dim C_\CC(\AA) = \dim D^\w(G)/\dim \AA = |G|^2/|A| = \dim D^\w(G,A).
  $$
   Therefore,
  $$
  C_\CC(\AA) = \Rep(D^\w(G,A)) \text{ and }  C_\CC(\Rep(D^\w(G,A))) = \AA\,.
  $$
   By Remark \ref{r:modularity}, $\AA$ is a modular category if, and only if, $\tilde p_\nu(A)$ is non-degenerate subgroup of $\SC(G,\w)$; this is equivalent to the assertion that the bicharacter $(\cdot|\cdot)_\nu$ on $A$ is non-degenerate.\  It follows from \cite[Thm. 3.2 and Cor. 3.5]{Muger3} that $\AA$ is modular if, and only if, $C_\CC(\AA)$ is modular.\ This proves the second assertion.
\end{proof}

The choice of cochain $\nu \in C^1(A, \hat{G})$ in the preceding theorem determines an embedding $\tilde p_\nu$ of $A$ into $\SC(G, \w)$.\ Therefore, the degeneracy of $\tilde p_\nu(A)$ in $\SC(G, \w)$ depends on the choice of $\nu$.\ However, the degeneracy of $\tilde p_\nu(A)$ can also be independent of the choice of $\nu$ in some situations.\ Important examples of this are contained in the next result. 

\begin{lem} If $A$ is an $\w$-admissible subgroup of $G$ such that $A\cong \BZ_2$ or $A \le [G,G]$.
 Then the bicharacter  $(\cdot|\cdot)_\nu$ on $A$ is independent of the choice of $\nu$.
\end{lem}
\begin{proof}
  Suppose  $\nu' \in C^1(A, \hat{G})$ is another cochain satisfying the condition of Theorem \ref{t:modular}.\ Then there is a group homomorphism $f: A \to \hat{G}$  such that $\nu'(a)(b) = f(a)(b)\nu(a)(b)$.\ Thus the associated bicharacter $(\cdot|\cdot)_{\nu'}$ is given by
  \begin{multline}
    (a|b)_{\nu'}=f(a)(b)\inv\nu(a)(b)\inv f(b)(a)\inv\nu(b)(a)\inv t_a(b) t_b(a) \\ = f(a)(b)\inv f(b)(a)\inv (a|b)_\nu.
  \end{multline}
  If $A \subseteq G'$, then $f(a)(b) = f(b)(a)=1$ for all $a, b \in A$, whence $(a|b)_{\nu}=(a|b)_{\nu'}$.
  
\medskip
  On the other hand, if $A$ is a group of order 2 generated by $z$ then $f(z)(z)^2=1$, so that
  $$
  (z,z)_{\nu'} = f(z)(z)^2 (z|z)_\nu = (z|z)_\nu\,. \qedhere
  $$
\end{proof}

\end{document}